\documentclass[11pt]{article}

\usepackage[a4paper,margin=1in]{geometry}
\usepackage[T1]{fontenc}
\usepackage[utf8]{inputenc}
\usepackage{lmodern}
\usepackage{microtype}
\usepackage{amsmath,amssymb,amsthm,mathtools}
\usepackage{graphicx}
\usepackage{booktabs}
\usepackage{array}
\usepackage{enumitem}
\usepackage{subcaption}
\usepackage{xcolor}
\usepackage{adjustbox}
\usepackage{hyperref}
\hypersetup{colorlinks=true,linkcolor=blue!50!black,citecolor=blue!50!black,urlcolor=blue!50!black}

\newtheorem{proposition}{Proposition}[section]
\newtheorem{theorem}[proposition]{Theorem}
\newtheorem{corollary}[proposition]{Corollary}
\newtheorem{lemma}[proposition]{Lemma}
\theoremstyle{definition}
\newtheorem{definition}[proposition]{Definition}
\newtheorem{problem}[proposition]{Problem}
\theoremstyle{remark}
\newtheorem{remark}[proposition]{Remark}
\newtheorem{observation}[proposition]{Observation}

\newcommand{\Ax}{\operatorname{Ax}}
\newcommand{\Sp}{\operatorname{Sp}}
\newcommand{\Fr}{\operatorname{Fr}}
\newcommand{\Sh}{\operatorname{Sh}}

\title{Directional Geometry and Anisotropy in the Partition Graph}
\author{Fedor B. Lyudogovskiy}
\date{}

\begin{document}
\maketitle

\begin{abstract}
We develop a directional formalism for the partition graph $G_n$ based on several canonical reference sets: the main chain, the self-conjugate axis, the spine, and the boundary framework. For each such set $S$, the graph distance $d_S$ induces a shell structure and a local trichotomy of edges into inward, outward, and level classes. Passing from edges to paths, we define directional corridors as monotone inward geodesics toward a chosen reference set and prove that every vertex admits at least one.

We then prove a structural non-equivalence theorem: for connected $G_n$, two nonempty reference sets induce the same edgewise directional field if and only if the difference of their distance functions is constant; in particular, distinct reference sets induce distinct directional fields. This provides a first precise formalization of anisotropy in $G_n$. We also show that every bounded neighborhood of a reference set is accessible by a monotone inward corridor, which gives a directional interpretation to previously established controlled regions around the axis, the spine, and the framework.

Finally, we complement the strict theory with a computational atlas illustrating edgewise directional statistics, directional mixing, local invariant drift, and corridor-based transport profiles.
\end{abstract}



\section{Introduction}

The partition graph $G_n$ is the graph whose vertices are the integer partitions of $n$ and whose edges correspond to elementary unit transfers between parts, followed by reordering into nonincreasing form; see Section~2 for a precise definition. In the preceding papers of this series, several geometric structures have been identified in $G_n$: the main chain, the self-conjugate axis, central neighborhoods, simplex layers, the degree landscape, the spine, the boundary framework, and rear morphology~\cite{LyuGrowing,LyuAxial,LyuSimplex,LyuDegree}. A persistent informal theme behind these constructions is that $G_n$ is not isotropic. Certain directions or transport regimes appear to be privileged, while others play visibly different roles.

The aim of this paper is to formalize that intuition. We do not seek a global geometric classification, nor do we attempt to reduce all motion in $G_n$ to a single dominant field. Instead, we develop a strict directional language relative to several canonical reference sets already present in the morphology of $G_n$: the main chain $M_n$, the self-conjugate axis $\Ax_n$, the spine $\Sp_n$, and the boundary framework $\Fr_n$. For each such set $S$, the graph distance $d_S$ produces a shell decomposition of $G_n$ and a local trichotomy of edges into inward, outward, and level classes. Passing from edges to paths, we define directional corridors as monotone inward geodesics toward a chosen reference set.

At the formal level, the paper has three main structural components. First, relative to any nonempty reference set $S \subseteq V(G_n)$, the distance function $d_S$ induces a local trichotomy of edges and a corresponding shell decomposition of $G_n$. Second, every vertex admits an $S$-monotone inward geodesic to $S$, which provides a pathwise notion of directional corridor. Third, if $G_n$ is connected, two nonempty reference sets induce the same edgewise directional field if and only if the difference of their distance functions is constant; in particular, distinct reference sets induce distinct directional fields. Together with the controlled-access principle for bounded neighborhoods, these results yield a first strict framework for directional geometry and anisotropy in the partition graph.

The contribution of the paper is therefore conceptual and organizational rather than enumerative. The relevant background comes from the global homotopy-theoretic study of $\mathrm{Cl}(G_n)$ and from the companion papers on local morphology, large-scale morphology, axial morphology, simplex stratification, and the degree landscape~\cite{LyuClique,LyuLocal,LyuGrowing,LyuAxial,LyuSimplex,LyuDegree}. The present paper does not re-establish those results; rather, it studies the directional relations among these structures. In particular, once a distinguished region is known to lie in a bounded neighborhood of the axis, the spine, or the framework~\cite{LyuAxial,LyuGrowing}, the corridor formalism immediately converts that containment statement into a directional-access statement. This is the main bridge between the present work and the earlier morphology papers.

We also include a compact computational atlas for the tested range $8 \le n \le 12$. The strict theory developed in Sections~2--5 already implies that different reference sets define different directional fields, but it does not by itself quantify the resulting anisotropy. The computations therefore serve two auxiliary purposes: they make the shell geometry visually explicit, and they measure how different directional regimes interact with local invariants such as degree, local simplex dimension, support size, and height.

The paper is organized as follows. Section~2 fixes the canonical reference sets and the basic distance observables. Section~3 introduces edgewise directional classes and the shell-based trichotomy. Section~4 passes from edges to paths and establishes the existence of monotone inward geodesics, which we call directional corridors. Section~5 contains the main structural anisotropy result and the controlled-access principle for bounded neighborhoods. Section~6 presents a computational directional atlas for the tested range $8 \le n \le 12$. The final section summarizes the resulting picture and formulates several open problems.

\section{Canonical reference sets and directional observables}

We now recall the objects that will serve as reference sets in the present paper. Let $V(G_n)$ denote the set of partitions of $n$ written in nonincreasing form. For standard background on integer partitions and Ferrers diagrams, see Andrews and Stanley~\cite{Andrews,Stanley}.

\begin{definition}\label{def:adjacency}
Let $\lambda=(\lambda_1,\dots,\lambda_k) \vdash n$. Choose two distinct indices $i,j$ with $1 \le i,j \le k$ and $\lambda_j>0$. Replace $\lambda_i$ by $\lambda_i+1$ and $\lambda_j$ by $\lambda_j-1$, delete the latter if it becomes $0$, and reorder the resulting parts into nonincreasing form. If the resulting partition is $\mu \ne \lambda$, we say that $\mu$ is obtained from $\lambda$ by an \emph{elementary unit transfer}.

Two partitions $\lambda,\mu \vdash n$ are adjacent in $G_n$ if one is obtained from the other by a single elementary unit transfer.
\end{definition}

This is the same adjacency convention used throughout the earlier papers of the series; in particular, it is the convention underlying the local transfer language and the height function~\cite{LyuClique,LyuLocal}.

\subsection{Canonical reference sets}

The \emph{main chain} is the hook chain
\[
M_n=\{(n-k,1^k): 0 \le k \le n-1\},
\]
joining the two antennas $(n)$ and $(1^n)$.

The \emph{self-conjugate axis} is
\[
\Ax_n=\{\lambda \vdash n : \lambda=\lambda'\}.
\]

For the \emph{boundary framework}, we use, following the large-scale morphology paper~\cite{LyuGrowing},
\[
\Fr_n = M_n \cup L_n \cup R_n,
\]
where
\[
L_n=\{(n-k,k): 1 \le k \le \lfloor n/2 \rfloor\}
\]
is the left edge and $R_n=L_n'$ is its conjugate right edge.

The \emph{spine} $\Sp_n$ is taken in the thin sense defined in the axial-morphology paper~\cite{LyuAxial}. Write the self-conjugate partitions of $n$ in decreasing lexicographic order as
\[
\sigma_1,\sigma_2,\dots,\sigma_m.
\]
For each consecutive pair $(\sigma_i,\sigma_{i+1})$, choose a shortest path in $G_n$ connecting them that is lexicographically minimal among all such shortest paths. The spine $\Sp_n$ is the union of $\Ax_n$ with all vertices lying on these chosen paths. This choice is not claimed to be the only reasonable one, but it is sufficiently canonical for the present directional theory and for the computations in Section~6.

\subsection{Distance observables, shells, and neighborhoods}

Let $S \subseteq V(G_n)$ be nonempty. We write
\[
d_S(\lambda)=\min_{\sigma \in S} d_{G_n}(\lambda,\sigma)
\]
for the graph distance from $\lambda$ to $S$.

For the four canonical reference sets we use the shorthand notation
\[
d_M,\qquad d_{\Ax},\qquad d_{\Sp},\qquad d_{\Fr}.
\]

\begin{definition}
Let $S \subseteq V(G_n)$ be nonempty and let $r \ge 0$. The $r$-th shell around $S$ is
\[
\Sh_S^{(r)}=\{\lambda \in V(G_n): d_S(\lambda)=r\},
\]
and the closed $r$-neighborhood of $S$ is
\[
N_S^{(\le r)}=\{\lambda \in V(G_n): d_S(\lambda)\le r\}.
\]
\end{definition}

We also use the height function
\[
h(\lambda)=\sum_{i \ge 1} i \lambda_i
\]
and the support size
\[
s(\lambda)=\ell(\lambda),
\]
that is, the number of positive parts of $\lambda$. These observables already play an important role in the local and degree-theoretic morphology of $G_n$~\cite{LyuLocal,LyuDegree}, and they will reappear in Section~6 as part of the directional drift profile.

\begin{remark}
The present paper uses \emph{relative} directional language. There is no absolute notion of inward motion. Instead, we speak about motion toward $M_n$, $\Ax_n$, $\Sp_n$, or $\Fr_n$, depending on which distance observable is being used.
\end{remark}

\subsection{Connectivity}

The directional-equivalence theorem in Section~5 relies on the fact that $G_n$ is connected.

\begin{proposition}\label{prop:Gn-connected}
For every $n \ge 1$, the partition graph $G_n$ is connected.
\end{proposition}

\begin{proof}
Let $\lambda=(\lambda_1,\dots,\lambda_k) \vdash n$, with $\lambda_1 \ge \cdots \ge \lambda_k > 0$. If $k=1$, then $\lambda=(n)$.

Assume $k \ge 2$. Choose the smallest part $\lambda_k$, and move its units one by one to the largest part $\lambda_1$. Each such step is an elementary unit transfer between two parts, followed by reordering, and therefore corresponds to an edge of $G_n$. After $\lambda_k$ transfers, the smallest part disappears, so the number of parts decreases by $1$.

Repeating the procedure, we eventually reach the one-part partition $(n)$. Thus every vertex is connected to $(n)$, and $G_n$ is connected.
\end{proof}

\section{Directional classes of moves}

We now formalize direction at the level of individual edges of $G_n$.

\subsection{Edgewise direction relative to a reference set}

\begin{proposition}\label{prop:S-trichotomy}
Let $S \subseteq V(G_n)$ be nonempty, and let $\lambda \sim \mu$ in $G_n$. Then
\[
|d_S(\lambda)-d_S(\mu)| \le 1.
\]
Equivalently, exactly one of the following three cases holds:
\[
d_S(\mu)=d_S(\lambda)-1,\qquad
d_S(\mu)=d_S(\lambda),\qquad
d_S(\mu)=d_S(\lambda)+1.
\]
\end{proposition}

\begin{proof}
Choose $\sigma \in S$ such that $d_S(\lambda)=d(\lambda,\sigma)$. Since $\lambda \sim \mu$,
\[
d(\mu,\sigma)\le d(\lambda,\sigma)+1=d_S(\lambda)+1,
\]
hence
\[
d_S(\mu)\le d_S(\lambda)+1.
\]
By symmetry,
\[
d_S(\lambda)\le d_S(\mu)+1.
\]
Therefore $|d_S(\lambda)-d_S(\mu)|\le 1$.
\end{proof}

\begin{definition}
Let $S \subseteq V(G_n)$ be nonempty, and let $\lambda \sim \mu$.
\begin{enumerate}[label=(\roman*)]
\item The edge $\lambda\mu$ is \emph{$S$-inward at $\lambda$} if
\[
d_S(\mu)=d_S(\lambda)-1.
\]
\item It is \emph{$S$-outward at $\lambda$} if
\[
d_S(\mu)=d_S(\lambda)+1.
\]
\item It is \emph{$S$-level at $\lambda$} if
\[
d_S(\mu)=d_S(\lambda).
\]
\end{enumerate}
\end{definition}

\begin{corollary}\label{cor:shell-adjacency}
Let
\[
\Sh_S^{(r)}=\{\lambda \in V(G_n): d_S(\lambda)=r\}.
\]
Then every edge of $G_n$ either lies inside a single shell $\Sh_S^{(r)}$ or joins two adjacent shells
$\Sh_S^{(r)}$ and $\Sh_S^{(r+1)}$.
\end{corollary}

\begin{proof}
Immediate from Proposition~\ref{prop:S-trichotomy}.
\end{proof}

\subsection{Tangential and transverse motion}

\begin{definition}
Let $S \subseteq V(G_n)$ be nonempty.
\begin{enumerate}[label=(\roman*)]
\item An edge $\lambda\mu$ is \emph{internal to $S$} if $\lambda,\mu \in S$.
\item It is \emph{transverse to the $S$-shell structure} if $d_S(\lambda)\ne d_S(\mu)$.
\item It is \emph{tangential to the $S$-shell structure} if $d_S(\lambda)=d_S(\mu)$.
\end{enumerate}
\end{definition}

\begin{proposition}\label{prop:tangential-transverse}
Let $S \subseteq V(G_n)$ be nonempty.
\begin{enumerate}[label=(\roman*)]
\item Every edge internal to $S$ is $S$-level.
\item Every $S$-inward or $S$-outward edge is transverse to the $S$-shell structure.
\item Every $S$-level edge is tangential to the $S$-shell structure.
\end{enumerate}
\end{proposition}

\begin{proof}
For (i), if $\lambda,\mu \in S$, then $d_S(\lambda)=d_S(\mu)=0$, so the edge is $S$-level. Statements (ii) and (iii) are immediate from the definitions.
\end{proof}

\begin{remark}
Tangential motion and internal motion should not be conflated. An edge can be tangential to the $S$-shell structure without being internal to $S$. For instance, this occurs when both endpoints lie in the same positive shell $\Sh_S^{(r)}$ with $r>0$.
\end{remark}

\subsection{Named directional classes and combined signatures}

We shall speak of \emph{axial}, \emph{spinal}, \emph{chain}, and \emph{framework} inward/outward/level motion according to whether $S$ equals $\Ax_n$, $\Sp_n$, $M_n$, or $\Fr_n$.

It is convenient to record several directional fields simultaneously.

\begin{definition}
For an oriented edge $(\lambda,\mu)$, define the combined directional signature
\[
\Sigma(\lambda,\mu)=\bigl(\sigma_M(\lambda,\mu),\sigma_{\Ax}(\lambda,\mu),\sigma_{\Sp}(\lambda,\mu),\sigma_{\Fr}(\lambda,\mu)\bigr),
\]
where
\[
\sigma_S(\lambda,\mu)=d_S(\mu)-d_S(\lambda)\in \{-1,0,+1\}.
\]
\end{definition}

\begin{definition}
An oriented edge is \emph{directionally coherent} if all nonzero components of its combined directional signature have the same sign. Otherwise it is called \emph{directionally mixed}.
\end{definition}

\begin{remark}
The same edge may be inward relative to one reference set, level relative to another, and outward relative to a third. This already shows that the directional geometry of $G_n$ is not governed by a single field.
\end{remark}

\section{Monotone paths and directional corridors}

The edgewise language of Section~3 gives rise to a natural path geometry relative to any chosen reference set $S$.

\subsection{Monotone paths}

\begin{definition}
A path
\[
P=(\lambda_0,\lambda_1,\dots,\lambda_m)
\]
in $G_n$ is called
\begin{enumerate}[label=(\roman*)]
\item \emph{$S$-monotone inward} if
\[
d_S(\lambda_{i+1})=d_S(\lambda_i)-1 \qquad (0 \le i < m);
\]
\item \emph{$S$-monotone outward} if
\[
d_S(\lambda_{i+1})=d_S(\lambda_i)+1 \qquad (0 \le i < m);
\]
\item \emph{$S$-level} if
\[
d_S(\lambda_{i+1})=d_S(\lambda_i) \qquad (0 \le i < m).
\]
\end{enumerate}
\end{definition}

\begin{lemma}\label{lem:monotone-length}
Let
\[
P=(\lambda_0,\lambda_1,\dots,\lambda_m)
\]
be an $S$-monotone inward path. Then
\[
d_S(\lambda_i)=d_S(\lambda_0)-i \qquad (0 \le i \le m).
\]
In particular,
\[
m \le d_S(\lambda_0).
\]
If moreover $\lambda_m \in S$, then
\[
m=d_S(\lambda_0),
\]
and the path is a geodesic from $\lambda_0$ to $S$.
\end{lemma}

\begin{proof}
The identity follows by induction on $i$. Since distances are nonnegative, $d_S(\lambda_0)-i \ge 0$, so $i \le d_S(\lambda_0)$ for all $i$, and in particular $m \le d_S(\lambda_0)$. If $\lambda_m \in S$, then $d_S(\lambda_m)=0$, so $0=d_S(\lambda_0)-m$, hence $m=d_S(\lambda_0)$. No path from $\lambda_0$ to $S$ can have length smaller than $d_S(\lambda_0)$.
\end{proof}

\begin{lemma}\label{lem:descent-neighbor}
Let $\lambda \in V(G_n)$ with $d_S(\lambda)>0$. Then $\lambda$ has a neighbor $\mu$ such that
\[
d_S(\mu)=d_S(\lambda)-1.
\]
\end{lemma}

\begin{proof}
Choose $\sigma \in S$ such that $d(\lambda,\sigma)=d_S(\lambda)$, and let
\[
\lambda=\lambda_0,\lambda_1,\dots,\lambda_r=\sigma
\]
be a shortest path from $\lambda$ to $\sigma$, where $r=d_S(\lambda)$. Then $\lambda_1$ is adjacent to $\lambda$, and
\[
d_S(\lambda_1)\le d(\lambda_1,\sigma)=r-1=d_S(\lambda)-1.
\]
By Proposition~\ref{prop:S-trichotomy},
\[
d_S(\lambda_1)\ge d_S(\lambda)-1.
\]
Hence $d_S(\lambda_1)=d_S(\lambda)-1$.
\end{proof}

\begin{proposition}\label{prop:monotone-geodesic}
For every vertex $\lambda \in V(G_n)$, there exists an $S$-monotone inward path from $\lambda$ to $S$ of length $d_S(\lambda)$. Equivalently, every vertex admits an $S$-monotone inward geodesic to $S$.
\end{proposition}

\begin{proof}
If $d_S(\lambda)=0$, there is nothing to prove. Otherwise Lemma~\ref{lem:descent-neighbor} gives a neighbor $\lambda_1$ with $d_S(\lambda_1)=d_S(\lambda)-1$. Repeating the argument inductively yields a path
\[
\lambda=\lambda_0,\lambda_1,\dots,\lambda_r
\]
such that
\[
d_S(\lambda_i)=d_S(\lambda)-i \qquad (0 \le i \le r).
\]
The process stops when $d_S(\lambda_r)=0$, that is, when $\lambda_r \in S$. Then $r=d_S(\lambda)$, and Lemma~\ref{lem:monotone-length} shows that the path is geodesic.
\end{proof}

\begin{corollary}\label{cor:shortest-to-S}
Let
\[
P=(\lambda_0,\lambda_1,\dots,\lambda_r)
\]
be any path from $\lambda_0$ to a vertex of $S$ of length $r=d_S(\lambda_0)$. Then $P$ is $S$-monotone inward.
\end{corollary}

\begin{proof}
Since $P$ has length $d_S(\lambda_0)$ and ends in $S$, it is a geodesic from $\lambda_0$ to $S$. For each $i$ we have
\[
d_S(\lambda_i)\le r-i,
\]
because the tail
\[
\lambda_i,\lambda_{i+1},\dots,\lambda_r
\]
connects $\lambda_i$ to $S$ in $r-i$ steps. On the other hand, the initial segment
\[
\lambda_0,\lambda_1,\dots,\lambda_i
\]
has length $i$, so
\[
d_S(\lambda_0)\le i+d_S(\lambda_i).
\]
Since $r=d_S(\lambda_0)$, it follows that
\[
r\le i+d_S(\lambda_i),
\]
and hence
\[
d_S(\lambda_i)=r-i
\qquad (0\le i\le r).
\]
Therefore
\[
d_S(\lambda_{i+1})=d_S(\lambda_i)-1
\qquad (0\le i<r),
\]
so $P$ is $S$-monotone inward.
\end{proof}

\subsection{Directional corridors and pathwise transport}

\begin{definition}
An \emph{$S$-directional corridor} is an $S$-monotone inward geodesic. In particular, we speak of axial, spinal, chain, and framework corridors when $S=\Ax_n$, $\Sp_n$, $M_n$, and $\Fr_n$, respectively.
\end{definition}

\begin{remark}
No uniqueness is claimed. A given vertex may admit many different directional corridors to the same reference set.
\end{remark}

\begin{definition}
Let $P=(\lambda_0,\dots,\lambda_m)$ be a path in $G_n$.
\begin{enumerate}[label=(\roman*)]
\item $P$ is \emph{internal to $S$} if $\lambda_i \in S$ for all $i$;
\item $P$ is \emph{shell-tangential with respect to $S$} if it is $S$-level;
\item $P$ is \emph{shell-transverse with respect to $S$} if every edge of $P$ is $S$-inward or $S$-outward.
\end{enumerate}
\end{definition}

\begin{proposition}\label{prop:internal-vs-level-paths}
Let $P$ be a path in $G_n$.
\begin{enumerate}[label=(\roman*)]
\item If $P$ is internal to $S$, then $P$ is $S$-level.
\item Every $S$-directional corridor is shell-transverse with respect to $S$.
\end{enumerate}
\end{proposition}

\begin{proof}
If every vertex of $P$ belongs to $S$, then every vertex has $d_S=0$, so every edge is $S$-level. If $P$ is an $S$-directional corridor, then every edge is $S$-inward by definition.
\end{proof}

\begin{remark}
An $S$-level path need not be internal to $S$. This happens whenever some positive shell $\Sh_S^{(r)}$, $r>0$, contains a nontrivial path.
\end{remark}

\section{Anisotropy and interaction with established morphology}

The results of this section isolate a basic strict form of anisotropy in $G_n$.

\subsection{Height orientation}

\begin{lemma}\label{lem:height-changes-on-edges}
If $\lambda \sim \mu$ in $G_n$, then
\[
h(\lambda)\neq h(\mu).
\]
In particular, every edge of $G_n$ admits a unique height direction.
\end{lemma}

\begin{proof}
By definition of the partition graph, adjacent vertices differ by a single elementary unit transfer between two parts, followed by reordering into nonincreasing form. Let $\mu$ be obtained from $\lambda$ by transferring one unit from the $j$-th part to the $i$-th part before reordering, where $i\ne j$.

For any partition $\xi=(\xi_1,\xi_2,\dots)$, write
\[
T_r(\xi):=\sum_{t\ge r}\xi_t.
\]
Then
\[
h(\xi)=\sum_{r\ge 1} T_r(\xi),
\]
since
\[
\sum_{r\ge 1} T_r(\xi)=\sum_{r\ge 1}\sum_{t\ge r}\xi_t=\sum_{t\ge 1} t\xi_t.
\]

Assume first that $i<j$. Before reordering, the modified sequence has the same total sum as $\lambda$, and its partial sums are larger than those of $\lambda$ for $k=i,\dots,j-1$ and unchanged otherwise. After reordering into nonincreasing form, the sums of the first $k$ terms can only increase. Hence
\[
\sum_{t=1}^k \mu_t\ge \sum_{t=1}^k \lambda_t
\qquad (k\ge 1),
\]
with strict inequality for some $k$. Since $|\lambda|=|\mu|$, this is equivalent to
\[
T_r(\mu)\le T_r(\lambda)
\qquad (r\ge 1),
\]
with strict inequality for some $r$. Therefore
\[
h(\mu)=\sum_{r\ge 1}T_r(\mu)<\sum_{r\ge 1}T_r(\lambda)=h(\lambda).
\]

The case $i>j$ is symmetric, and yields $h(\mu)>h(\lambda)$. In either case,
\[
h(\lambda)\ne h(\mu).
\]
Thus adjacent vertices always have distinct heights.
\end{proof}

\begin{proposition}\label{prop:height-acyclic}
Orient each edge of $G_n$ from the endpoint of smaller height to the endpoint of larger height. Then the resulting orientation is acyclic.
\end{proposition}

\begin{proof}
By Lemma~\ref{lem:height-changes-on-edges}, this orientation is well defined on every edge. Along every oriented edge the height strictly increases. Therefore a directed cycle would force
\[
h(\lambda_0)<h(\lambda_1)<\cdots<h(\lambda_k)=h(\lambda_0),
\]
which is impossible.
\end{proof}

\begin{remark}
The height orientation is intrinsic: it does not depend on a chosen reference set. It should therefore be compared with the shell-based directional fields rather than identified with any one of them.
\end{remark}

\subsection{Directional equivalence and its failure}

\begin{definition}
Let $S,T \subseteq V(G_n)$ be nonempty. We say that $S$ and $T$ are \emph{directionally equivalent} if
\[
\sigma_S(\lambda,\mu)=\sigma_T(\lambda,\mu)
\]
for every oriented edge $(\lambda,\mu)$ of $G_n$.
\end{definition}

\begin{theorem}\label{thm:directional-equivalence}
Assume that $G_n$ is connected. Let $S,T \subseteq V(G_n)$ be nonempty. Then the following are equivalent:
\begin{enumerate}[label=(\roman*)]
\item $S$ and $T$ are directionally equivalent;
\item the function
\[
d_S-d_T
\]
is constant on $V(G_n)$.
\end{enumerate}
\end{theorem}

\begin{proof}
If $d_S-d_T$ is constant, then for every oriented edge $(\lambda,\mu)$,
\[
d_S(\mu)-d_S(\lambda)=d_T(\mu)-d_T(\lambda),
\]
so $S$ and $T$ are directionally equivalent.

Conversely, assume that $\sigma_S(\lambda,\mu)=\sigma_T(\lambda,\mu)$ for every oriented edge. Fix a base vertex $\lambda_0 \in V(G_n)$ and let $\lambda \in V(G_n)$ be arbitrary. Since $G_n$ is connected by Proposition~\ref{prop:Gn-connected}, there exists a path
\[
\lambda_0,\lambda_1,\dots,\lambda_m=\lambda.
\]
Summing the edge increments along this path yields
\[
d_S(\lambda)-d_S(\lambda_0)
=\sum_{i=0}^{m-1}\sigma_S(\lambda_i,\lambda_{i+1})
=\sum_{i=0}^{m-1}\sigma_T(\lambda_i,\lambda_{i+1})
=d_T(\lambda)-d_T(\lambda_0).
\]
Hence
\[
d_S(\lambda)-d_T(\lambda)=d_S(\lambda_0)-d_T(\lambda_0),
\]
which is independent of $\lambda$.
\end{proof}

\begin{corollary}\label{cor:distinct-not-equivalent}
Let $S,T \subseteq V(G_n)$ be nonempty and distinct. Then $S$ and $T$ are not directionally equivalent. Equivalently, there exists an oriented edge $(\lambda,\mu)$ such that
\[
\sigma_S(\lambda,\mu)\ne \sigma_T(\lambda,\mu).
\]
\end{corollary}

\begin{proof}
Suppose that $S$ and $T$ are directionally equivalent. By Theorem~\ref{thm:directional-equivalence}, the function
\[
d_S-d_T
\]
is constant on $V(G_n)$; write this constant as $c$.

Choose $s\in S$ and $t\in T$. Since $d_S(s)=0$ and $d_T(t)=0$, we have
\[
c=d_S(s)-d_T(s)=-d_T(s)\le 0
\]
and
\[
c=d_S(t)-d_T(t)=d_S(t)\ge 0.
\]
Hence $c=0$, so
\[
d_S=d_T
\]
on all of $V(G_n)$.

Now let $s\in S$. Then $d_S(s)=0$, hence $d_T(s)=0$, which means that $s\in T$. Thus $S\subseteq T$. By symmetry, $T\subseteq S$. Therefore $S=T$, contradicting the assumption that the two sets are distinct.
\end{proof}

\begin{remark}
Applying Corollary~\ref{cor:distinct-not-equivalent} to
\[
M_n,\qquad \Ax_n,\qquad \Sp_n,\qquad \Fr_n,
\]
we conclude that the four canonical directional fields are pairwise distinct. This is the basic strict form of anisotropy in $G_n$.
\end{remark}

\subsection{Access to bounded neighborhoods}

The corridor formalism extends from reference sets to their bounded neighborhoods.

\begin{proposition}\label{prop:reach-neighborhood}
Let $S \subseteq V(G_n)$ be nonempty, and let $r \ge 0$. For every vertex $\lambda \in V(G_n)$, there exists a path from $\lambda$ to the closed neighborhood
\[
N_S^{(\le r)}=\{\mu \in V(G_n): d_S(\mu)\le r\}
\]
which is $S$-monotone inward and has length
\[
\max\{d_S(\lambda)-r,0\}.
\]
In particular, if $d_S(\lambda)\ge r$, then there exists an $S$-monotone inward geodesic from $\lambda$ to $N_S^{(\le r)}$ of length $d_S(\lambda)-r$.
\end{proposition}

\begin{proof}
If $d_S(\lambda)\le r$, the trivial path suffices. Otherwise choose an $S$-directional corridor
\[
\lambda=\lambda_0,\lambda_1,\dots,\lambda_m
\]
from $\lambda$ to $S$, where $m=d_S(\lambda)$ by Proposition~\ref{prop:monotone-geodesic}. By Lemma~\ref{lem:monotone-length},
\[
d_S(\lambda_i)=d_S(\lambda)-i.
\]
Setting $i_*=d_S(\lambda)-r$, we obtain $d_S(\lambda_{i_*})=r$, so $\lambda_{i_*} \in N_S^{(\le r)}$. The initial segment of the corridor has length $d_S(\lambda)-r$ and is optimal, since each edge changes $d_S$ by at most $1$.
\end{proof}

\begin{definition}
A set $X \subseteq V(G_n)$ is \emph{$S$-controlled with radius $r$} if
\[
X \subseteq N_S^{(\le r)}.
\]
\end{definition}

\begin{proposition}\label{prop:controlled-access}
Let $X \subseteq V(G_n)$ be $S$-controlled with radius $r$. Then every vertex $\lambda \in V(G_n)$ admits an $S$-monotone inward geodesic to the controlled region $N_S^{(\le r)}$. If $d_S(\lambda)\ge r$, the length of such a geodesic is $d_S(\lambda)-r$.
\end{proposition}

\begin{proof}
Immediate from Proposition~\ref{prop:reach-neighborhood}.
\end{proof}

\begin{remark}
Proposition~\ref{prop:controlled-access} is intentionally formulated in abstract form. Its role is to convert previously established containment statements into directional ones. Whenever a distinguished family of vertices is already known to lie in a bounded neighborhood of $\Ax_n$, $\Sp_n$, or $\Fr_n$, the proposition immediately yields corresponding axial, spinal, or framework corridors to that region.
\end{remark}

\section{Computational directional atlas}

We now complement the strict theory with a small computational atlas. All computations in this section were carried out for the tested range $8 \le n \le 12$. The figures use fixed layouts for $G_{10}$ and $G_{12}$, while the tables summarize either selected values of $n$ or the whole tested range.

For each $n$, we enumerated all partitions of $n$, constructed the graph $G_n$ using Definition~\ref{def:adjacency}, and computed the reference sets $M_n$, $\Ax_n$, $\Sp_n$, and $\Fr_n$. We then computed the distance fields $d_M$, $d_{\Ax}$, $d_{\Sp}$, and $d_{\Fr}$. Local degrees are graph degrees. The local simplex dimension of a vertex was computed as the maximal clique size containing that vertex minus one. Canonical directional corridors were obtained by choosing, at each step, the lexicographically smallest neighbor that decreases the relevant distance by one.

\subsection{Edgewise directional statistics}

Figures~\ref{fig:axial-shells} and~\ref{fig:spinal-shells} show the vertex shell geometry of $G_{10}$ relative to the axis and the spine, using the same embedding in both panels. The two shell structures are very close for $G_{10}$, so the vertices at which $d_{\Ax}$ and $d_{\Sp}$ differ are marked by black rings. Figures~\ref{fig:axial-edges} and~\ref{fig:framework-edges} display the corresponding same-shell versus cross-shell edge behavior.

\begin{figure}[t]
\centering
\begin{subfigure}[t]{0.48\textwidth}
\centering
\includegraphics[width=\textwidth]{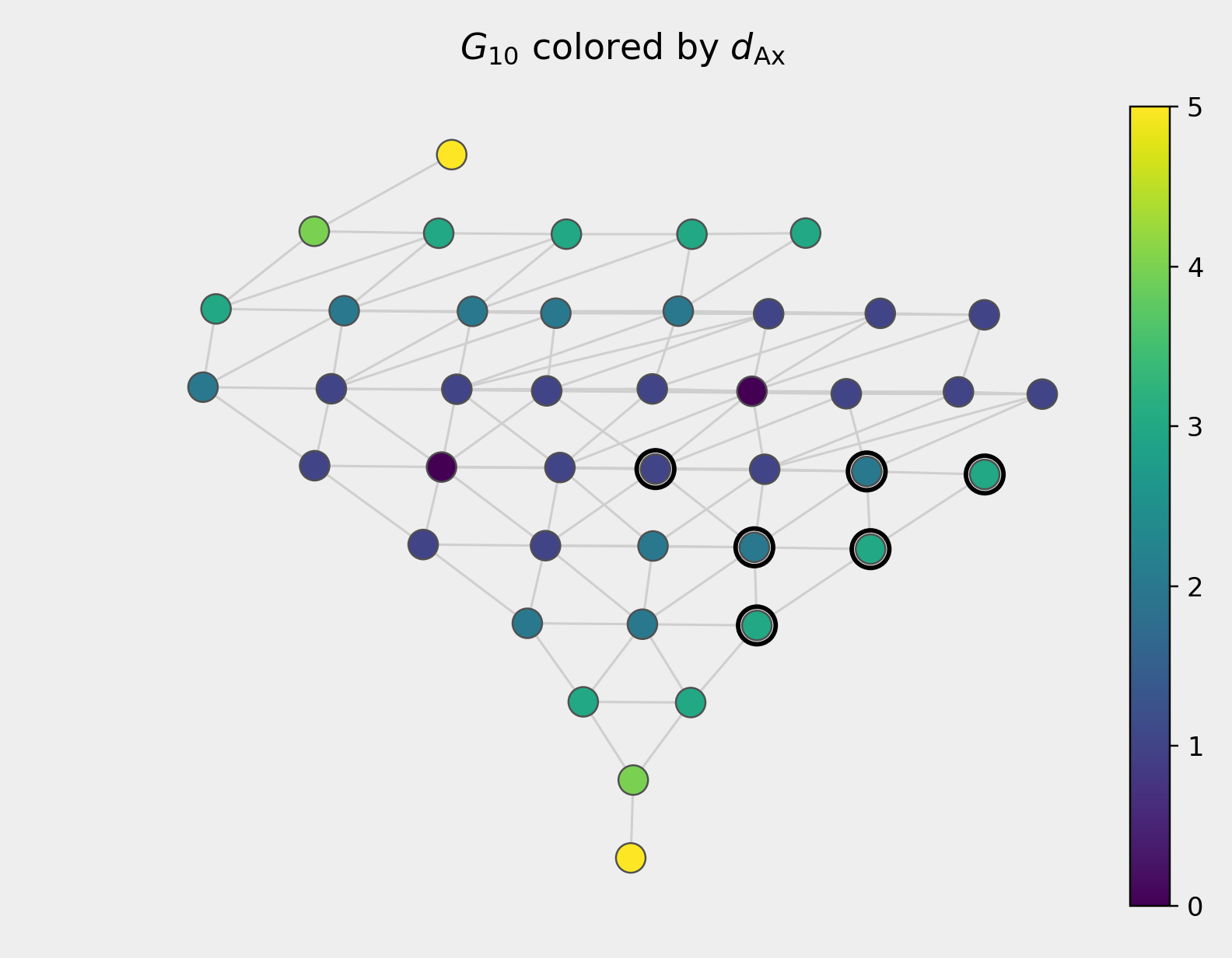}
\caption{Coloring by $d_{\Ax}$; black rings mark vertices where $d_{\Ax}\neq d_{\Sp}$.}
\label{fig:axial-shells}
\end{subfigure}
\hfill
\begin{subfigure}[t]{0.48\textwidth}
\centering
\includegraphics[width=\textwidth]{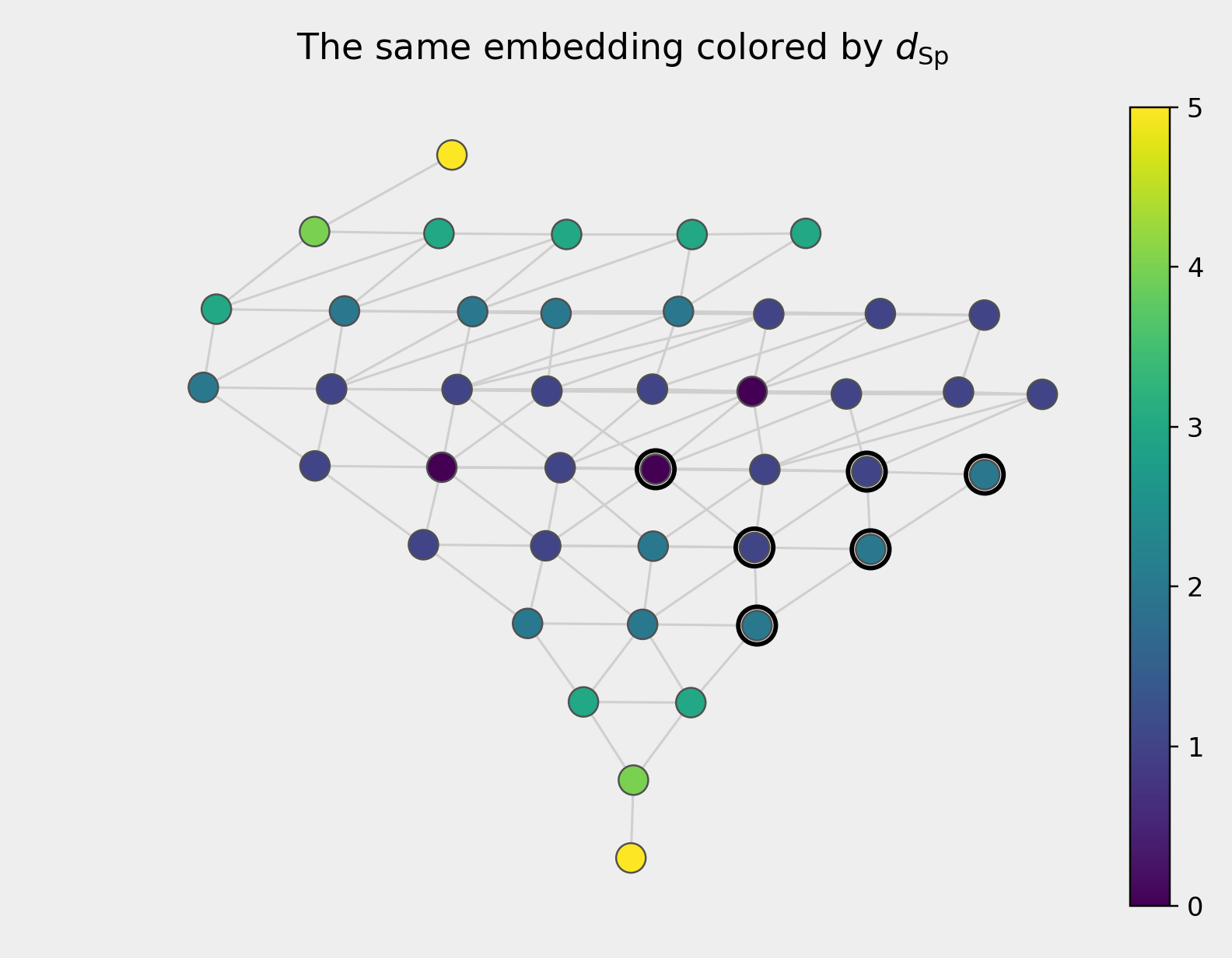}
\caption{Coloring by $d_{\Sp}$; black rings mark vertices where $d_{\Ax}\neq d_{\Sp}$.}
\label{fig:spinal-shells}
\end{subfigure}
\caption{Two shell structures on the same embedding of $G_{10}$, shown with identical vertex positions to isolate the difference between axial and spinal shells. The two shell structures are close but not identical; vertices where the two shell distances differ are marked by black rings.}
\end{figure}

\begin{figure}[t]
\centering
\begin{subfigure}[t]{0.48\textwidth}
\centering
\includegraphics[width=\textwidth]{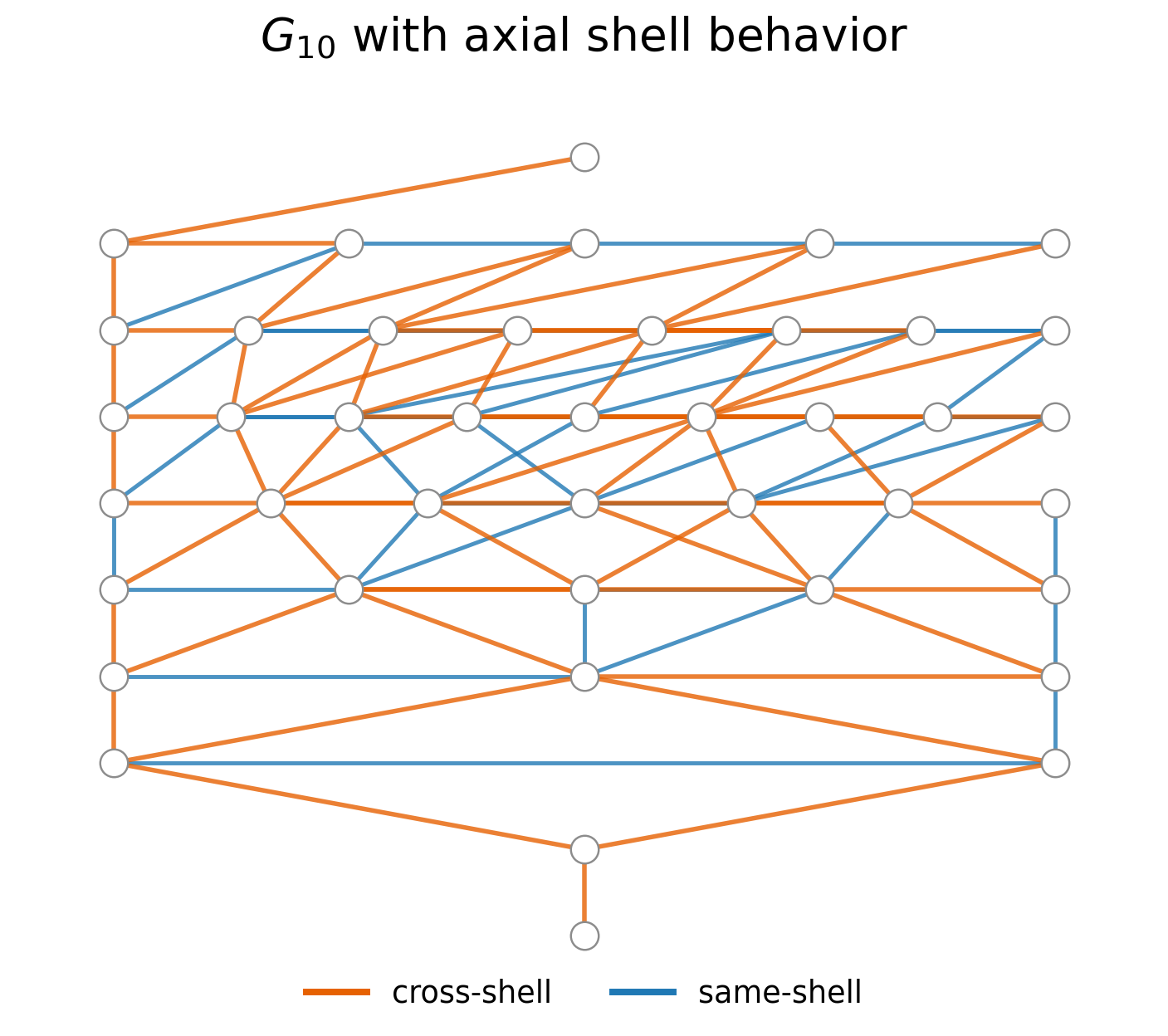}
\caption{Relative to $\Ax_{10}$.}
\label{fig:axial-edges}
\end{subfigure}
\hfill
\begin{subfigure}[t]{0.48\textwidth}
\centering
\includegraphics[width=\textwidth]{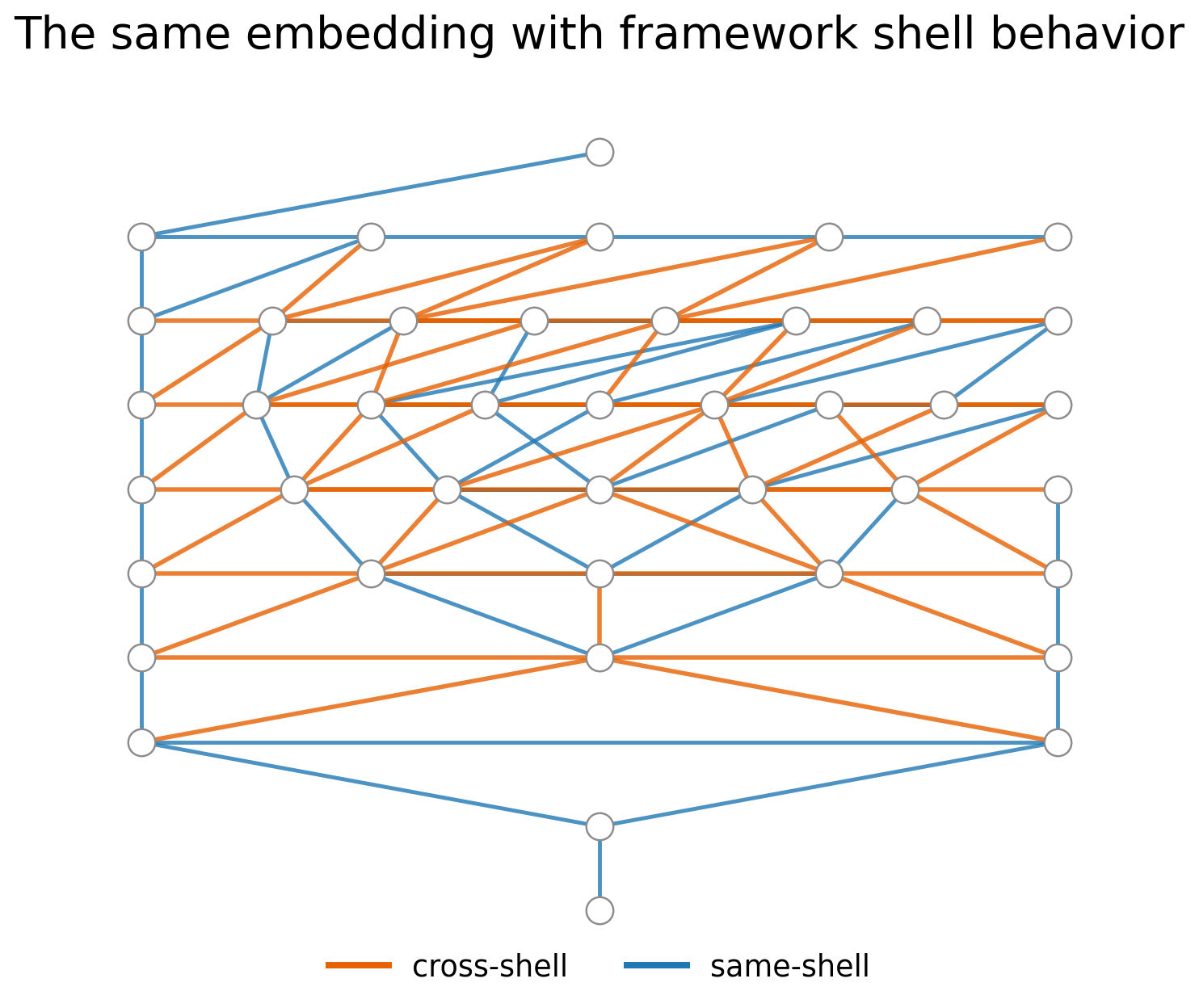}
\caption{Relative to $\Fr_{10}$.}
\label{fig:framework-edges}
\end{subfigure}
\caption{Same-shell and cross-shell edge behavior for two different reference sets. In each panel, blue edges remain inside a shell and orange edges cross between adjacent shells.}
\end{figure}

\begin{table}[t]
\centering
\small
\begin{tabular}{ccccc}
\toprule
$n$ & reference set & level share & transverse share & max shell radius \\
\midrule
8 & $M_n$ & 36.2\% & 63.8\% & 3 \\
8 & $\Ax_n$ & 36.2\% & 63.8\% & 4 \\
8 & $\Sp_n$ & 34.0\% & 66.0\% & 4 \\
8 & $\Fr_n$ & 48.9\% & 51.1\% & 2 \\
10 & $M_n$ & 35.1\% & 64.9\% & 4 \\
10 & $\Ax_n$ & 40.4\% & 59.6\% & 5 \\
10 & $\Sp_n$ & 40.4\% & 59.6\% & 5 \\
10 & $\Fr_n$ & 45.6\% & 54.4\% & 3 \\
12 & $M_n$ & 34.4\% & 65.6\% & 6 \\
12 & $\Ax_n$ & 42.3\% & 57.7\% & 6 \\
12 & $\Sp_n$ & 40.3\% & 59.7\% & 6 \\
12 & $\Fr_n$ & 41.5\% & 58.5\% & 4 \\
\bottomrule
\end{tabular}
\caption{Edgewise directional statistics for the canonical reference sets.}
\label{tab:edgewise-directional-stats}
\end{table}

\begin{observation}\label{obs:edgewise-nonuniformity}
In the tested range, the shell geometry around the canonical reference sets is visibly non-uniform. In particular, the level/transverse split depends substantially on the reference set. For example, the framework exhibits a larger level share than the main chain, the axis, or the spine at $n=8$ and $n=10$, while the axial and spinal shell geometries are visually distinguishable in the figures, with the local discrepancies marked explicitly in Figure~\ref{fig:axial-shells}--\ref{fig:spinal-shells}.
\end{observation}

\subsection{Combined signatures and directional mixing}

\begin{table}[t]
\centering
\small
\begin{tabular}{cccc}
\toprule
$n$ & distinct combined signatures & mixed-edge share & coherent-edge share \\
\midrule
8 & 24 & 42.6\% & 57.4\% \\
9 & 32 & 41.1\% & 58.9\% \\
10 & 32 & 41.2\% & 58.8\% \\
11 & 33 & 43.5\% & 56.5\% \\
12 & 34 & 37.2\% & 62.8\% \\
\bottomrule
\end{tabular}
\caption{Combined directional signatures in the tested range.}
\label{tab:combined-signatures}
\end{table}

\begin{observation}\label{obs:directional-mixing}
Directionally mixed edges form a substantial part of $G_n$ throughout the tested range. The mixed-edge share remains between $37.2\%$ and $43.5\%$ for $8 \le n \le 12$, and the number of distinct combined signatures increases from $24$ at $n=8$ to $34$ at $n=12$.
\end{observation}

\begin{remark}
Observation~\ref{obs:directional-mixing} is the most direct empirical reflection of Corollary~\ref{cor:distinct-not-equivalent}. The theorem says that different reference sets cannot induce the same field globally; the computations show that their disagreement is not confined to a small exceptional subset of edges.
\end{remark}

\subsection{Directional drift of local invariants}

For every oriented edge $(\lambda,\mu)$ we computed
\[
\Delta \deg = \deg(\mu)-\deg(\lambda),\qquad
\Delta \dim_{\mathrm{loc}} = \dim_{\mathrm{loc}}(\mu)-\dim_{\mathrm{loc}}(\lambda),
\]
together with $\Delta h$. Table~\ref{tab:directional-drift} aggregates these quantities over all oriented edges of the tested range $8 \le n \le 12$.

\begin{table}[t]
\centering
\scriptsize
\setlength{\tabcolsep}{4pt}
\begin{tabular}{lrrrrrr}
\toprule
directional class & mean $\Delta \deg$ & median $\Delta \deg$ & share $\Delta \deg>0$ & mean $\Delta \dim_{\mathrm{loc}}$ & share $\Delta \dim_{\mathrm{loc}}>0$ & mean $\Delta h$ \\
\midrule
axial inward & 1.81 & 2 & 65.5\% & 0.34 & 38.6\% & -0.98 \\
spinal inward & 2.10 & 2 & 71.6\% & 0.33 & 35.7\% & -0.60 \\
framework inward & -1.23 & -1 & 27.0\% & -0.42 & 3.2\% & 0.93 \\
axial level & 0.00 & 0 & 35.0\% & 0.00 & 8.7\% & 0.00 \\
spinal level & 0.00 & 0 & 34.6\% & 0.00 & 12.3\% & 0.00 \\
framework level & 0.00 & 0 & 29.6\% & 0.00 & 7.3\% & 0.00 \\
\bottomrule
\end{tabular}
\caption{Empirical drift of local invariants along selected directional classes, aggregated over $8 \le n \le 12$.}
\label{tab:directional-drift}
\end{table}

\begin{observation}\label{obs:drift-contrast}
The empirical drift of local invariants depends strongly on the directional regime. In the tested range, axial inward and spinal inward motion have positive average drift in both degree and local simplex dimension, while framework inward motion has negative average drift in both quantities. Numerically, the mean degree drift is $1.81$ for axial inward edges, $2.10$ for spinal inward edges, and $-1.23$ for framework inward edges.
\end{observation}

\begin{remark}
The level classes have zero mean drift by symmetry, but their positive-increment shares remain informative. In particular, the positive degree share is about $35\%$ for axial and spinal level motion, compared with $29.6\%$ for framework level motion on the tested range.
\end{remark}

\subsection{Corridor profiles}

We next pass from individual edges to canonical corridors. Figure~\ref{fig:corridors} shows three corridor types from the same starting vertex in $G_{12}$.

\begin{figure}[t]
\centering
\includegraphics[width=0.94\textwidth]{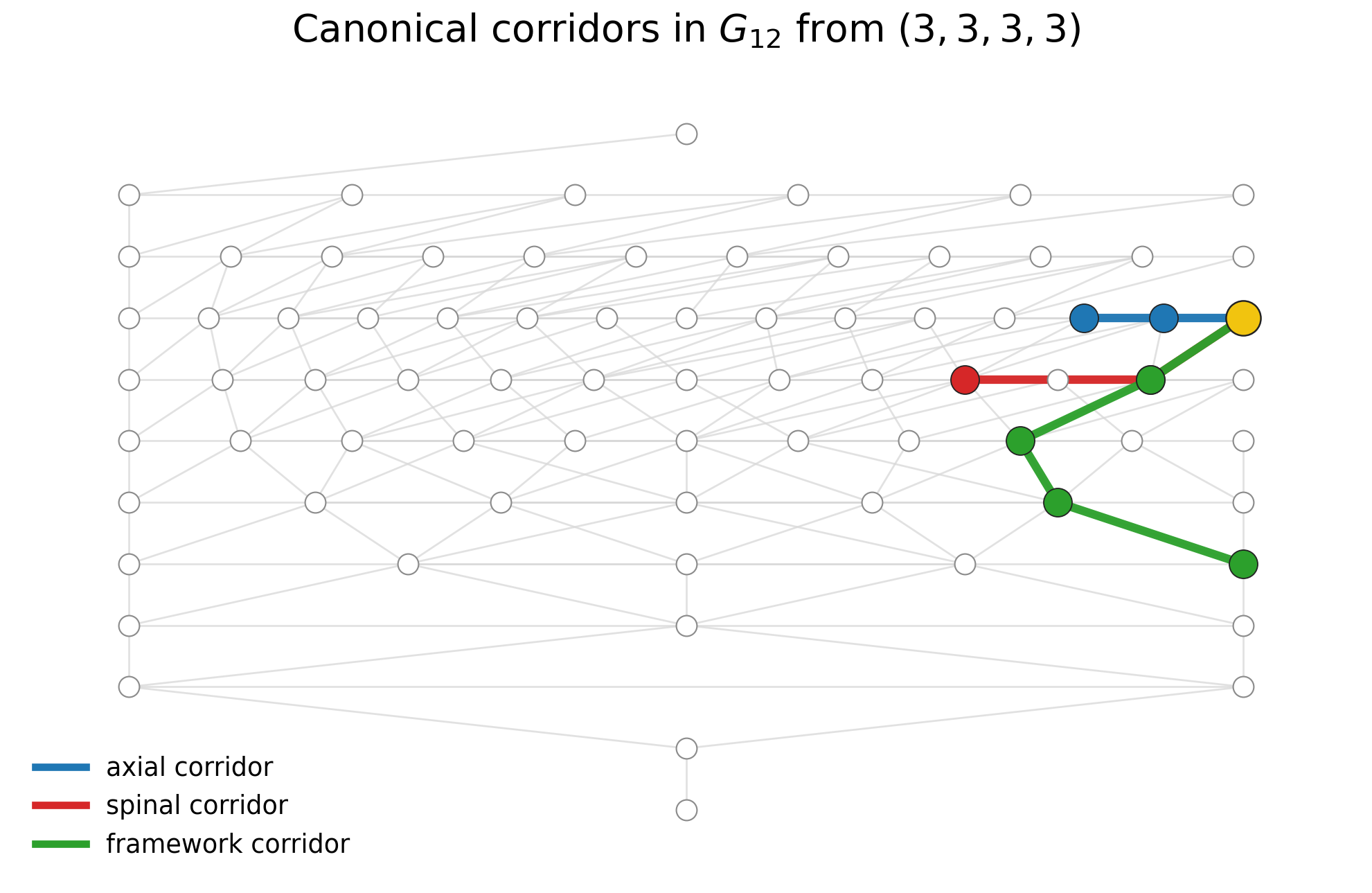}
\caption{Canonical axial, spinal, and framework corridors from the same starting vertex in $G_{12}$, defined by the lexicographically minimal descent rule.}
\label{fig:corridors}
\end{figure}

Table~\ref{tab:corridor-profiles} summarizes median hit times for several targets. The columns ``to $\Ax_n$'' and ``to $\Sp_n$'' measure the median number of corridor steps needed to hit the corresponding reference set. For the complexity-based targets, hit times are computed only over successful starts, while the accompanying success column records the proportion of starts whose canonical corridor reaches the target at all. The column ``to top degree'' records the median first-hit time for the maximum-degree zone, together with its success rate along the chosen corridor family. The final two columns do the same for the top local-simplex-dimension zone.

\begin{table}[t]
\centering
\small
\begin{tabular}{ccrrrrrr}
\toprule
$n$ & corridor type & to $\Ax_n$ & to $\Sp_n$ & to top degree & success & to top simplex & success \\
\midrule
8 & axial & 2 & 1.5 & 2 & 63.6\% & 1 & 95.5\% \\
8 & spinal & 1.5 & 1 & 1.5 & 36.4\% & 1 & 95.5\% \\
8 & framework & 0 & 0 & 0 & 4.5\% & 0 & 36.4\% \\
10 & axial & 2 & 2 & 1 & 57.1\% & 0 & 100.0\% \\
10 & spinal & 1 & 2 & 1 & 31.0\% & 0 & 100.0\% \\
10 & framework & 0.5 & 0 & 0 & 2.4\% & 0 & 57.1\% \\
12 & axial & 2 & 2 & 2 & 51.9\% & 1 & 63.6\% \\
12 & spinal & 2 & 2 & 2 & 22.1\% & 1 & 53.2\% \\
12 & framework & 1 & 1 & 0.5 & 2.6\% & 0 & 14.3\% \\
\bottomrule
\end{tabular}
\caption{Corridor access statistics for selected geometric and complexity targets.}
\label{tab:corridor-profiles}
\end{table}

\begin{observation}\label{obs:corridor-contrast}
Canonical corridors toward different reference sets are visibly non-coincident, both geometrically and combinatorially. In particular, framework corridors remain strongly boundary-confined: in Table~\ref{tab:corridor-profiles} their success rate for the top-degree zone is below $5\%$ at each of $n=8,10,12$, while axial corridors reach that zone with success rates between $51.9\%$ and $63.6\%$.
\end{observation}

\begin{observation}\label{obs:interior-access}
In the tested range, corridors directed toward axial and spinal neighborhoods provide more direct access to structured interior regions than framework corridors. This is visible both in Figure~\ref{fig:corridors} and in the corridor statistics: axial and spinal corridors always reach the corresponding central neighborhoods by construction, whereas framework corridors reach top-complexity zones only rarely.
\end{observation}

\subsection{Summary of empirical anisotropy}

The computations support the following qualitative picture. First, the canonical directional fields are visibly non-coincident, both edgewise and pathwise. Second, local complexity drift depends on the directional regime: axial and spinal inward motion tend to move toward higher-complexity regions, while framework inward motion tends to move away from them. Third, access to structured interior regions is directionally uneven: corridor families relative to the axis, the spine, and the framework do not play interchangeable roles.

These observations do not replace the strict theory of Sections~2--5. Rather, they show how that theory manifests itself concretely in finite partition graphs and provide an initial measured atlas of anisotropy in $G_n$.

\section{Conclusion and open problems}

We developed a directional formalism for the partition graph $G_n$ based on graph distance from canonical reference sets. This formalism yields inward, outward, and level edge classes, shell decompositions, and monotone inward geodesics, which we interpret as directional corridors. The central structural result is the non-equivalence theorem: distinct nonempty reference sets induce distinct directional fields. Together with the intrinsic acyclic height orientation, this gives a first precise formal meaning to the statement that the partition graph is anisotropic.

A second, conceptual contribution of the paper is the controlled-access principle. It shows that any region known to lie in a bounded neighborhood of the axis, the spine, or the framework automatically inherits corresponding directional corridors. In this way, earlier morphological results can be read in directional terms without being reproved here.

The computational atlas complements the strict theory. It does not prove universal optimality statements, but it shows that the directional language introduced in this paper captures visible and measurable structure already in the tested range of values of $n$.

We conclude with several open problems.

\begin{problem}[Quantitative anisotropy]
Develop numerical invariants measuring the extent of anisotropy in $G_n$. Natural candidates include the mixed-edge share, divergences between directional drift profiles, and distances between corridor families associated with different reference sets.
\end{problem}

\begin{problem}[Directional access to complexity]
Determine whether high-complexity regions of $G_n$, such as high-degree zones or top simplex layers, are asymptotically more accessible through certain directional regimes than through others.
\end{problem}

\begin{problem}[Compatibility of directional fields]
Study the existence of paths that are simultaneously well behaved with respect to several reference sets. For example, when can a path be spinal-inward while also nonincreasing in axial distance?
\end{problem}

\begin{problem}[Rear-directed geometry]
Clarify whether rear-directed transport defines a genuinely new directional field or whether it can be reduced to the existing axis/spine/framework language.
\end{problem}

\begin{problem}[Directional morphology across $n$]
Extend the present fixed-$n$ theory to the comparative-growth picture across $n$. How do directional fields behave under natural overlays $G_n \to G_{n+k}$? Which directional patterns stabilize, and which remain local to small or intermediate values of $n$?
\end{problem}

\begin{problem}[Toward a discrete transport theory]
Develop a broader framework in which directional corridors, shell crossings, and controlled neighborhoods are treated as parts of a discrete transport geometry on $G_n$.
\end{problem}

\section*{Acknowledgements}

The author acknowledges the use of ChatGPT (OpenAI) for discussion, structural planning, and editorial assistance during the preparation of this manuscript. All mathematical statements, proofs, computations, and final wording were checked and approved by the author, who takes full responsibility for the contents of the paper.

\end{document}